\documentclass{amsart}
\newcommand{\Ps}{\mathbf{P}}
\newcommand{\As}{\mathbf{A}}
\newcommand{\Q}{\mathbf{Q}}

\newcommand{\Z}{\mathbf{Z}}

\newcommand{\oK}{\overline{K}}
\newcommand{\cO}{\mathcal{O}}
\newcommand{\cC}{\mathcal{C}}

\newcommand{\cB}{\mathcal{B}}
\newcommand{\cX}{\mathcal{X}}
\newcommand{\cU}{\mathcal{U}}
\newcommand{\cW}{\mathcal{W}}
\DeclareMathOperator{\rank}{rank}

\DeclareMathOperator{\Gal}{Gal}

\DeclareMathOperator{\NS}{NS}

\DeclareMathOperator{\pr}{pr}

\DeclareMathOperator{\spa}{sp}
\DeclareMathOperator{\Aut}{Aut}
\DeclareMathOperator{\tw}{tw}

\newtheorem{lemma}{Lemma}[section]
\newtheorem{proposition}[lemma]{Proposition}
\newtheorem{theorem}[lemma]{Theorem}
\newtheorem{corollary}[lemma]{Corollary}
\newtheorem{conjecture}[lemma]{Conjecture}
\theoremstyle{definition}
\newtheorem{notation}[lemma]{Notation}
\newtheorem{definition}[lemma]{Definition}
\newtheorem{example}[lemma]{Example}
\theoremstyle{remark}

\newtheorem{remark}[lemma]{Remark}

\title[Average Mordell-Weil rank of elliptic surfaces]{The average Mordell-Weil rank of elliptic surfaces over number fields}
\author{Remke Kloosterman}
\email{klooster@math.unipd.it}
\address{Universit\`a degli Studi di Padova,
Dipartimento di Matematica ``Tullio Levi-Civita",
Via Trieste 63,
35121 Padova, Italy}
\thanks{The author would like to thank Matthias Sch\"utt for comments on a previous version of this paper.}
\date{\today}
\begin{document}
\begin{abstract} 
 Let $K$ be a finitely generated field over $\Q$. 
 Let $\cX\to \cB$ be a family of elliptic surfaces over $K$ such that each elliptic fibration has the same configuration of singular fibers. 
 Let $r$ be the minimum of the Mordell-Weil rank in this family. Then we show that the locus inside $|\cB|$ where the Mordell-Weil rank is at least $r+1$ is a sparse subset.
 
 In this way we  prove Cowan's conjecture on the average Mordell-Weil rank of elliptic surfaces over $\Q$ and prove a similar result for elliptic surfaces over arbitrary number fields.
\end{abstract}
\maketitle
\section{introduction}
In a recent preprint Alex Cowan \cite{Cowan} formulated the following conjecture on the average rank of elliptic curves over $\Q(t)$.
Let $\mu$ be the Mahler measure  on $\Z[t]$ and let 
$P_d(M ) = \{p \in \Z[T ] \mid \deg(p) \leq  d, \mu(p) < M \}$.
Define $S_{m,n}(M)$ to be the set
\[ \left\{E_{A,B} : y^2 = x^3 + A(t )x + B(t ) \left|\begin{array}{c} A \in P_m(M^2), B \in P_n(M ^3),\\ 4A(t )^3 + 27B(t )^2\neq 0\end{array}\right.\right\}.\]
\begin{conjecture}[Cowan]
 For every pair of positive integers $m,n$ we have 
\[ \lim_{M\to\infty} \frac{1}{\# S_{m,n}(M)} \sum_{E\in S_{m,n}(M)} \rank E(\Q(t))=0.\]
\end{conjecture}
 Battistoni, Bettin and Delaunay proved this conjecture for $1\leq m,n\leq 2$ and for certain unirational subfamilies of $S_{2,2}$, see \cite{BBD}. 

In this paper we will prove Cowan's conjecture for all $(m,n)$ with $m,n\geq 1$, and generalize this to arbitrary number fields. The main ingredient holds in a more general context:
Fix a field $K$, finitely generated over $\Q$. Fix a smooth geometrically irreducible base variety $\cB/K$. 
Let $\cX\to \cB$ be a family of elliptic surfaces with a section, with $\cC\to \cB$ the base curve of the elliptic fibration (cf. Section~\ref{secFam}). Let $\eta$ be the generic point of $\cB$. Then the generic fiber of $\cX_\eta\to\cC_\eta$ is an elliptic curve  $E_\eta/ K(\eta)(\cC_\eta)$. Similarly, for a closed point $b\in |\cB|$ the generic fiber of $\cX_b\to\cC_b$ is an elliptic curve  $E_b/K(b)(C_b)$. 
There is a specialization map
\[ E_\eta(K(\eta)(\cC_\eta))\to E_b(K(b)(\cC_b))\]
and a second map if one passes to algebraic closures of $K(\eta)$ and $K(b)$. We prove the following result on these specialization maps:
\begin{theorem} Let $K$ be a finitely generated field over $\Q$, and let $\mathcal{X}\to \mathcal{B}$ be a family of elliptic surfaces with a section over $K$. Then there exists a sparse set $Z\subset |\mathcal{B}|$ such that for all $b\in |\cB|\setminus Z$  
the specialisation maps
\[ E_{\eta} (\overline{K(\eta)}(\cC_\eta)) \to E_b(\overline{K(b)}(\cC_b))\]
and
\[ E_{\eta} ({K(\eta)}(\cC_\eta)) \to E_b({K(b)}(\cC_b))\]
are bijective.
%
%
\end{theorem}
(This result is a combination of Proposition~\ref{prpMain} and Theorem~\ref{mainThmB}.) For a definition of thin subsets and of sparse subsets see Section~\ref{secFam}.
The existence of the first isomorphism relies heavily on  \cite[Theorem 8.3]{MauPoo} and the main result of \cite{And}. We obtain the second isomorphism from  the first by using an appropriate form of Hilbert irreducibility theorem.

In the proof of  Cowan's conjecture we need to take for $\cB$ some Zariski open subset of $S_{m,n}$. Let $\cX$ be the universal elliptic curve over $\cB$. In order to prove Cowan's conjecture it suffices to show that the associated generic Mordell-Weil group $E_{\eta}(K(\eta)(C_\eta))$  is finite . 
If $m=4k$ and $n=6k$ for some integer $k>1$ then using Noether-Lefschetz theory \cite{CoxNL} one can show that  $E_{\eta}(\overline{K(\eta)}(C_\eta))$ is trivial. This yields Cowan's conjecture almost immediately. 
 However, for certain values of $(m,n)$ we were not able to determine $E_{\eta}(\overline{K(\eta)}(C_\eta))$ and, moreover, in the range $1\leq m\leq 4; 1\leq n\leq 6$ this larger group has positive rank. We will use quadratic twisting to show that for  all choices of $m$ and $n$ there is a $b\in |\cB|\setminus Z$ such that the arithmetic Mordell-Weil group $E_{b} ({K(b)}(\cC_b))$ is finite.
More precisely, given  a rational point $b\in \mathcal{B}(K)$ consider 
\[ \tw(b):=\left\{ d\in K^*/(K^*)^2 \left| \begin{array}{c}\mbox{ there exists a point  } b'\in B(K)\\ \mbox{ such that }  \cX_b/\cC_b\not \cong \cX_{b'}/\cC_{b'} \\\mbox { and }
( \cX_b/\cC_b)_{K(\sqrt{d})} \cong (\cX_{b'}/\cC_{b'})_{K(\sqrt{d})} \end{array}\right.\right\} \]
the set of $d$ such that the quadratic twist of $\cX_b$ by $\sqrt{d}$ is also contained in the family $\cX\to \cB$.
 \begin{theorem}
  Let $K$ be a finitely generated field over $\Q$ and let $\mathcal{X}\to \mathcal{B}$ be a family of elliptic surfaces with a section.

Let $b\in B(K)$ be a point such that $\tw(b)$ is infinite and such that the configuration of singular fibers of $\cX_b\to \cC_b$ is the same as $\cX_\eta\to \cC_\eta$.
Then $E_\eta(K(\eta)(C_\eta))$ is finite.
\end{theorem}

These two results combined show that on the complement of a thin subset of $S_{m,n}$ the Mordell-Weil rank is zero. This is sufficient to prove Cowan's conjecture:
\begin{corollary}Let $K$ be a number field.
Fix positive integers $m,n$.
The set of ellliptic surfaces in $S_{m,n}$  with Mordell-Weil rank zero has density one.
\end{corollary}

\begin{corollary}Let $K=\Q$.
Fix positive integers $m,n$.
Then the average rank of elliptic surfaces in $S_{m,n}$ equals zero.
\end{corollary}

We will also prove a generalization of Cowan's conjecture to arbitrary number fields,  see Corollary~\ref{corCowGen}. However, in this case we have to adjust the definition of $S_{m,n}$ in order to exclude trivial elliptic fibrations, i.e., elliptic surfaces which are birational to $E\times \Ps^1$ as fibered surfaces.

The organization of this paper is as follows. In Section~\ref{secFam} we recall some standard results on families of elliptic surfaces. Moreover, we use one of the results from \cite{MauPoo} to determine the locus in the family where the specialization map on the Mordell-Weil group is injective, and then to describe the locus where this map is bijective.
In Section~\ref{secCow} we study the universal Weierstrass equation over $S_{m,n}$ and use the results from the previous section to show that for $K=\Q$ the average rank is zero.

\section{Families of elliptic surfaces}\label{secFam}

Let $K$ be a field. 
\begin{definition} An \emph{elliptic surface} $\pi: X\to C$ over $K$ consists of a geometrically irreducible smooth projective surface  $X/K$, a geometrically irreducible smooth projective curve $C/K$  and a flat $K$-morphism $\pi:X\to C$ such that the generic fiber of $\pi$ is a smooth projective curve of genus 1, and none of the fibers of $\pi$ contains a $(-1)$-curve.

An \emph{elliptic surface with a section} is an elliptic surface $\pi:X\to C$ together with a section $\sigma_0:C\to X$, defined over $K$.

Let $E/K(C)$ be the generic fiber. Then the group $E(K(C))$ of $K(C)$-valued points of $E$  can be naturally identified the set of rational sections of $\pi$. This latter set can be into a group by fiberwise addition, where the zero element of the fiber over a point $p\in C$ is $\sigma_0(p)$.

We call an elliptic surface $\pi:X\to C$ \emph{trivial} if there is an elliptic curve $E/K$ and a birational map $\psi:X\to E\times C$ such that $\pi=\pr_2\circ \psi$ as rational maps.
\end{definition}

Since $C$ is a smooth curve we can extend every rational section to a section of $\pi$, hence $E(K(C))$ is also the set of sections of $\pi$.
\begin{definition}Let $S$ be a  smooth projective surface over an algebraically closed field $K$. Denote with $\NS(S)$ the N\'eron-Severi group of $S$, the group of divisors on $S$ modulo algebraic equivalence.
\end{definition}

\begin{definition} Let $\pi:X\to C$ be an elliptic surface with a section. 
Let $T\subset \NS(X_{\oK})$ be the \emph{trivial subgroup}, generated by the irreducible components (over $\oK$) of the singular fibers not intersecting $\sigma_0(C)$, the class of a smooth fiber and the image of the zero section. (\cite[Section 6.1]{SSbook})
\end{definition}

We have the following results, which contains both the Mordell-Weil theorem and  the Shioda-Tate formula:
\begin{proposition}  \label{prpST} Let $\pi:X\to C$ be an elliptic surface with a section. If $X$ is not a trivial elliptic surface then there is a natural isomorphism of groups
\[ \NS(X_{\oK})/T\cong E(\oK(C)).\]
In particular, $E(\oK(C))$ is finitely generated.
\end{proposition}
\begin{proof}
See \cite[Theorem 6.5]{SSbook}
\end{proof}
\begin{notation}
Let $\cB$ be a $K$-variety. Then denote with $|\cB|$ the set of closed points of $\cB$ and with $\eta$ the generic point of $\cB$, with residue field $K(\eta)=K(\cB)$
\end{notation}
\begin{definition}\label{defFamily}
A \emph{family of elliptic surfaces with a section} consists of a smooth geometrically irreducible $K$-variety $\cB$, a smooth projective surface $\mathcal{X}\to \cB$, and a smooth curve $\mathcal{C}\to \cB$ together with morphisms $\pi: \mathcal{X}\to \mathcal{C}$ and $\sigma_0:\mathcal{C}\to \mathcal{X}$ of $\cB$-schemes, such that for each closed point $b\in |\cB|$ we have that $\pi_b:\mathcal{X}_b\to \mathcal{C}_b$ is an elliptic surface over $K(b)$ with zero-section $(\sigma_0)_b:\mathcal{C}_b\to \mathcal{X}_b$.
\end{definition}
\begin{remark}
Let $\cX\to \cC\to \cB$ be a family of elliptic surfaces.
Then $\cX_\eta$ is an elliptic surface over $K(\eta)$, and  for every $b\in |\cB|$ the surface $\cX_b$ is an elliptic surface over $K(b)$. We denote with $\cX_{\overline{\eta}}$ the base change of $\cX_\eta$ to $\overline{K(\eta)}$ and with $\cX_{\overline{b}}$ the base change of $\cX_b$ to $\overline{K(b)}$.
\end{remark}
\begin{proposition} Let $\cX\to \cC\to \cB$ be a family of elliptic surfaces.
Suppose $\cX_\eta$ is not a trivial elliptic surface. Then for all $b\in B$ we have that $\cX_b$ is not a trivial elliptic surface.
\end{proposition}
\begin{proof} Recall that $\cX\to \cB$  and $\cC\to \cB$ are both  smooth morphism. By \cite[Lemma IV.1.1]{MiES}, a trivial elliptic surface has $b_1(\cX_b)=2g(\cC_b)+2$, whereas a nontrivial elliptic surface has $b_1(\cX_b)=2g(\cC_b)$. Since both $g(\cC_b)$ and $b_1(\cX_b)$ are constant on $\cB$ we find that either all closed fibers are trivial and so is  the generic fiber trivial or none of the closed fibers is trivial.
\end{proof}

We have the following result from Maulik and Poonen \cite[Proposition 3.6(a)]{MauPoo}:
\begin{proposition}
With the same notation as before, we have that the specialization map
\[ \spa: \NS(\mathcal{X}_{\overline{\eta}})\to \NS(\mathcal{X}_{\overline{b}})\]
is injective with torsion-free cokernel.
\end{proposition}

One easily shows the following
\begin{lemma} \label{lemTri}
With the same notation as before, we have that $\spa(T_\eta)\subset T_b$.
\end{lemma}
\begin{proof}
This specialization map obviously maps the class of the zero section to the class of the section zero and the class of a fiber to a class of a fiber. Moreover, the specialiatzion of a fiber component is contained in a single fiber of the specialized surface, hence $\spa(T_\eta)\subset T_b$.
\end{proof}

The following statements can be found in \cite[Chapter 5 and Section 6.1]{SSbook}:
Let $\pi: X\to C$ be an elliptic surface over a field $K$. Let $\Delta$ be the discriminant. For $p\in \Delta({\oK})$ consider the dual graph of $\pi^{-1}(p)$ and eliminate the component intersecting the image of the zero section. Call this graph $\Gamma_p$. Then this graph is of type $A_n,D_m, E_6,E_7$ or $E_8$. Let $\Lambda_p$ be the associated lattice. Then
\[ T=\Lambda'\oplus^{\perp}_{p\in \Delta({\oK})} \Lambda_p\]
where $\Lambda'$ is rank two lattice generated by  classes $F$ and $Z$, with $F^2=0$, $Z^2=-\chi(\cO_X)$, $F.Z=1$.

\begin{definition} Let $\pi: X \to C$ be an elliptic surface with section over a field $K$. The \emph{configuration of singular fibers} is a finite  multiset $M$, whose elements are formal symbols $A_n$, with $n\in \Z_{>0}$, $D_m$, with $m\in \Z_{>0}$, $E_6$, $E_7$, $E_8$, such that the multiplicity of $\Lambda$ in $M$ equals the number of $p\in C(\overline{K})$ such that $\Gamma_p$ is isomorphic to the graph with label $\Lambda$.
\end{definition}

\begin{lemma}\label{lemSpeInj} Let $\cB'$ be the subscheme of $\cB$ such that for all $b\in |\cB'|$ the configuration of singular fibers of $\pi_b:\cX_b\to \cC_b$ equals the configuration of singular fibers of $\pi_\eta:\cX_\eta\to \cC_\eta$. Then $\cB'$ is nonempty open and for all closed points $b\in |\cB'|$ we have 
\[ \spa(T_\eta)=\spa(T_b.)\]
In particular, we have that for all $b\in |\cB'|$ the specialization maps
\[E_{\eta}({\overline{K(\eta)}}(C_\eta)) \to E_b(\overline{K(b)}(C_b)) \mbox{ and }E_{\eta}({K(\eta)}(C_\eta)) \to E_b({K(b)}(C_b))\]
are injective.
\end{lemma}
\begin{proof}
Let $b\in |\cB'|$ then $\spa(T_\eta)$ is a sublattice of $T_b$ by the previous lemma. 
Since the configuration of singular fibers are the same we have that $\rank(T_b)=\rank (\spa(T_\eta))$, hence image has finite index, and that $\det(T_b)=\det(\spa(T_\eta))$, hence this index is one. In particular, $\spa(T_\eta)=T_b$ for all $b\in \cB'$.

Using Tate's algorithm \cite{Tate} one easily sees that the type of singular fibers can be determined by the valuation of three standard invariants of a Weierstrass equation $(v(j), v(c_4),v(c_6))$. In particular, the locus $\cB'$ is nonempty and open in $\cB$. 
Moreover,  for all $b \in |\cB|\setminus (Z'\cup Z'')$  we have the following chain of morphisms
\[ E_\eta(\overline{K(\eta)}(C_\eta)) \stackrel{\sim}{\longrightarrow} \NS(\mathcal{X}_{\overline{\eta}})/T_\eta \hookrightarrow \NS(\mathcal{X}_{\overline{b}})/T_b \stackrel{\sim}{\longrightarrow} E_b(\overline{K(b)}(C_b)).\]
This shows that the first (geometric) specialization map is injective. The second (arithmetic) specialization map is the restriction of the first one and is  also injective.
\end{proof}

\begin{definition}  Let $K$ be a finitely generated field over $\Q$. Let $\cB$ be a $K$-variety. Call a subset $S$ of $|\cB|$ \emph{sparse} if there exists a dominant and generically finite morphism $\pi:\cB_0\to \cB$ of irreducible $K$-varieties, such that for each $s\in S$  the fiber $\pi^{-1}(s)$ is empty or contains at least two closed points.
\end{definition}

\begin{definition} Let $K$ be a field of characteristic zero. Let $V/K$ be a  $K$-variety. A subset $S$ of $V(K)$ is called a \emph{thin subset of type $I$} if $S$ is contained in a Zariski closed subset of $V(K)$.
A subset $S$ of $V(K)$ is a called a \emph{thin subset of type II}, if there exists another $K$-variety $V'$ such that $\dim V=\dim V'$ and a finite morphism $\varphi : V'\to V$ of degree at least 2, such that $S\subset \varphi(V'(K))$.

A subset $S$ of $V(K)$ is \emph{thin} if it is a subset of a finite union of thin subsets of type I and type II.
\end{definition}
\begin{remark}
If $S$ is  a sparse subset of $|\As^n|$ then $S\cap \As^n(K)$ is  a thin set.
\end{remark}

\begin{proposition}\label{prpMain}  Let $K$ be a finitely generated field over $\Q$.
Then there is a sparse subset $Z\subset |\cB|$ such that for each $b\in |\cB|\setminus Z$ the specialization map
\[ E_{\eta}(\overline{K(\eta)}(C_\eta)) \to E_b(\overline{K(b)}(C_b))\]
is an isomorphism.
\end{proposition}
\begin{proof}
Let $\cB'$ be as in the previous lemma. Let $Z'=|\cB|\setminus| \cB'|$.
Since $\mathcal{X}\to B$ is smooth and projective, we have by \cite[Proposition 3.6 and Theorem 8.3]{MauPoo} that there is a sparse subset $Z''\subset |\cB|$, such that
\[ \spa: \NS(\mathcal{X}_{\overline{\eta}})\to \NS(\mathcal{X}_{\overline{b}})\]
is an isomorphism for all $b\in |\cB|\setminus Z''$.
Then for all $b \in |\cB|\setminus (Z'\cup Z'')$  we have isomorphisms
\[ E_\eta(\overline{K(\eta)}(C_\eta)) \stackrel{\sim}{\longrightarrow} \NS(\mathcal{X}_{\overline{\eta}})/T_\eta \stackrel{\sim}{\longrightarrow} \NS(\mathcal{X}_{\overline{b}})/T_b \stackrel{\sim}{\longrightarrow} E_b(\overline{K(b)}(C_b)).\]
Since $Z'$ is a Zariski-closed proper subset of $|\cB|$ and $Z''$ is sparse it follows that $Z'\cup Z''$ is sparse in $|\cB|$.
\end{proof}

We will now prove a similar statement for the specialization on the arithmetic Mordell-Weil groups, $E_{\eta}({K(\eta)}(C_\eta)) \to E_b({K(b)}(C_b))$. For this we aim to compare the $\Gal(\overline{k(\eta)}/k(\eta))$-action on $ E_{\eta}(\overline{K(\eta)}(C_\eta))$ and the $\Gal(\overline{k(b)}/k(b))$-action on $E_b(\overline{K(b)}(C_b))$.

\begin{remark}Fix a finite Galois extension $L$ of $K(\eta)=K(\cB)$ with Galois group $G$. Let $L_0$ be the algebraic closure of $K$ in $L$. Suppose first that $L_0 \neq L$.

Then $L=K(\cB')$, for some integral $K$-variety $\cB'$, which may be not geometrically integral. From the fact that $L_0\neq L$ it follows that  there is a finite rational map $\tau:\cB'\to\cB$ of degree at least 2. By shrinking $\cB$ and $\cB'$ if necessary we may assume that for $\tau$ is an unramified finite flat morphism.
In particular, for every $p\in |\cB|$ we have that $\tau^{-1}(p)$ consists of finitely many closed points  $p_1,\dots,p_t$. Each point $p_i$ yields  a unramified field extension $K(p_i)/K(p)$. Now, the locus where $t>1$ holds,  defines a sparse subset $S$ of $|\cB|\setminus Z$.

Hence for each point $p\in |\cB|\setminus (Z\cup S)$ we have that $t=1$ and that $K(p_1)/K(p)$ is Galois with group $G$. In this case we call $K(p_1)/K(p)$  the specialization of the Galois extension $L/K(\eta)$.

If $L_0=L$ then for every $b\in |\cB|$ we have that the extension $L(b)/K(b)$ is  Galois with group $G$.
\end{remark}

\begin{theorem} \label{mainThmB}
 Let $K$ be a finitely generated field over $\Q$.
 Let $\cX\to \cB$ be a family of elliptic surfaces with base curve $\cC$. Then there is a sparse subset $Z\subset |\cB|$, such that for every $b\in |\cB|\setminus Z$ the specialization map
\[ E_\eta(K(\eta)(\cC_\eta)) \to E_b(K(b)(\cC_b))\]
is an isomorphism.
\end{theorem}

\begin{proof}
Let $Z'$ be the sparse subset of $|\cB|$ such that on $|\cB|\setminus Z'$ the geometric specialization map 
\[ E_\eta(\overline{K(\eta)}(\cC_\eta)) \to E_b(\overline{K(b)}(\cC_b))\]
is an isomorphism.

Recall that the group $E_\eta(\overline{K(\eta)}(\cC_\eta))$ is finitely generated. Moreover, for each $P\in E_\eta(\overline{K(\eta)}(\cC_\eta))$ we have that $K(\eta)(\cC_\eta)(P)\subset \overline{K(\eta)}(\cC_\eta)$ is algebraic and finitely generated over $K(\eta)(\cC_\eta)$. In particular, there exists a minimal finite extension $L$ of $K(\eta)$ such that $L/K(\eta)$ is Galois and
\[ E_\eta(L(\cC_\eta))=E_\eta(\overline{K(\eta)}(\cC_\eta)).\]

Let $G=\Gal(L/K(\eta))$. Then there is a thin set $Z'$ such that for all $b\in |\cB|\setminus Z'$ the specialization $M$ of $L$ to $K(b)$ exists. Then $M/K(b)$ is Galois with group $G$ and $E_b(M(\cC_b))=E_b(\overline{K(b)}(\cC_b))$. 

Let $P_1,\dots,P_s$ be generators for $E_\eta({K(\eta)}(\cC_\eta))$. Let $H$ be a subgroup of the group $\Gal(L/K(\eta))$. If $E(\overline{K(\eta)}(\cC_\eta))^H = E_\eta(K(\eta)(\cC_\eta))$ then we call $H$ irrelevant, otherwise we call $H$ relevant.

 For every  relevant subgroup $H$ of $\Gal(L/K(\eta))$, let \[Q_{H,1},\dots Q_{H,t_H}\in E(\overline{K(\eta)}(\cC_\eta))^H\] be points such that 
$P_1,\dots,P_s,Q_{H,1},\dots, Q_{H,t_H}$ generate $E(\overline{K(\eta)}(C_\eta))^H$.

For a relevant subgroup $H$ of $G$ let $Z_H\subset |\cB|$ be the locus of all $b\in |\cB|$ such that $\spa(Q_{H,i})\in K(b)$ for all $i=1\dots,t_H$. Then for every $b\in |\cB|\setminus Z_H$ we have that $E_b(K(b)(\cC_b))\neq E_\eta(L(\cC_\eta))^H$.
Let \[Z=Z'\bigcup_{H \mathrm{ relevant}} Z_H.\]
Then for all $b \in |\cB|\setminus Z$ we have $E_b(K(b)(\cC_b))=E_\eta(L(\cC_\eta))^H$ for some subgroup $H$ of $G$ but since $b\not \in Z$, this group $H$ is irrelevant. In particular,
\[ E_b(K(b)(\cC_b))=E_\eta(L(\cC_\eta))^H=E_\eta(K(\eta)(\cC_\eta))\]
From the construction of $Z_H$ it follows that $Z_H$ is sparse and therefore $Z$ is sparse, since is a finite union  of sparse subsets.
\end{proof}

\section{Cowan's conjecture}\label{secCow}
In this section we want to use Theorem~\ref{mainThmB} to prove Cowan's conjecture and generalize this result to number fields. As explained in the Introduction we will use quadratic twists for this.

\begin{proposition} \label{prpTwist}Let $V$ be a $\Q$-vector space. Let $\rho: \Gal(\oK/K)\to \Aut(V)$ be a Galois representation, such that $\rho$ factors through a finite group.
For any $d\in  K^*/(K^*)^2$, let $\chi_d$ be the quadratic character associated with the field extension $K(\sqrt{d})/K$. Let $\rho^{(d)}:\Gal(\oK/K)\to \Aut(V)$ be the twisted Galois representation $\rho^{(d)}(g)=\rho(g) \circ (\chi_d(g) Id_V)$

Then there exists at most $\dim V$ many $d\in K^*/(K^*)^2$ such that $\rho^{(d)}$ has an invariant subspace.
\end{proposition}
\begin{proof}
Suppose first that $V$ is the trivial representation then for all $d\neq 1$ we have that $\rho^{(d)}(\cdot)=\chi_d(\cdot) Id_V$ has no invariant subspace. Similarly, if $V$ has a twist which is trivial then all other twists  (including the trivial one) have no invariant subspace.

We proceed now by induction on $\dim V$. The case $\dim V=1$ is covered by the above.
Suppose now that $\dim V>1$. Suppose that for some $d$ there is an invariant subspace then by Maschke's theorem we can decompose $V^{(d)}=\Q \oplus V'$. Then there are at most $d-1$ twists of $V'$ which have an invariant subspace and only the trivial twist of $\Q$ has an invariant subspace. 
\end{proof}

Consider now a family of elliptic surfaces $\cX\to\cC\to \cB$.
Given  a rational point $b\in B(K)$ consider 
\[ \tw(b):=\left\{ d\in K^*/(K^*)^2 \left| \begin{array}{c}\mbox{ there exists a point  } b'\in B(K)\\ \mbox{ such that }  \cX_b/\cC_b\not \cong \cX_{b'}/\cC_{b'} \\\mbox { and }
( \cX_b/\cC_b)_{K(\sqrt{d})} \cong (\cX_{b'}/\cC_{b'})_{K(\sqrt{d})} \end{array}\right.\right\} \]
the set of $d$ such that the quadratic twist of $\cX_b$ by $\sqrt{d}$ is also contained in the family $\cX\to \cB$.
 \begin{proposition}\label{prpRankVan}
  Let $K$ be a finitely generated field over $\Q$ and let $\mathcal{X}\to \mathcal{B}$ be a family of non-trivial elliptic sufaces with a section.

Let $b\in B(K)$ be a rational point such that $\tw(b)$ is infinite and the configuration of singular fibers of $\cX_b\to \cC_b$ is the same as $\cX_\eta\to \cC_\eta$.
Then $E_\eta(K(\eta)(\cC_\eta))$ is finite.
\end{proposition}
\begin{proof}
Our assumption on the configuration of singular fibers and Lemma~\ref{lemSpeInj} yeild  that the specialization map
\[ E_\eta(K(\eta)(\cC_\eta)) \to E_b(K(\cC_b))\]
is injective.
Moreover, the assumption on the configuration of singular fibers is formulated over $\oK$, hence for all $d\in \tw(b)$ we have that 
\[ E_\eta(K(\eta)(\cC_\eta)) \to E_b^{(d)}(K(\cC_b))\]
is injective.

Consider now the $\Gal(\oK/K)$-representation $V=E_{\eta}(\oK(t))\otimes_\Z \Q$. 
The Galois representation on $ E^{(d)}_{b}(K(\cC_b))\otimes_\Z \Q$ is precisely $V^{(d)}$. Hence if the rank of $E^{(d)}_{b}(K(\cC_b))$ is positive then $V^{(d)}$ has an invariant subspace. Since there are only finitely many $d\in K^*/(K^*)^2$ with this property it follows that there is a $d\in \tw(b)$ with $E^{(d)}_{b}(K(\cC_b))$ finite and therefore $E_\eta(K(\eta)(\cC_\eta))$ is finite.
\end{proof}

\begin{notation}
For a commutative ring $R$ let $R[t]_{ d}$ be the $R$-module of polynomials in $t$ with coefficients from $R$ and of degree at most $d$. Let $m,n$ be integers and let
\[S_{m,n}(R)=\{(A,B)\in R[t]_{m}\times R[t]_n | 4A^3+27B^2\neq 0\}.\]
Let $k=\lceil \max(m/4,n/6) \rceil$. Suppose now that $R$ is an integral domain not of characteristic 2 or 3. Let $K$ be its quotient field. To a pair $(A,B)\in S_{m,n}$ we can associate an elliptic curve $E_{(A,B)}/K(t)$ with Weierstrass equation
\[ y^2=x^3+A(t)x+B(t).\]
Moreover we can associate a hypersurface $W_{A,B}$ in in the $\Ps^2$-bundle $\Ps(\cO\oplus \cO(-2k)\oplus \cO(-3k))$ over $\Ps^1$, given by
\[ -Y^2Z+X^3+s^{4k}A\left(\frac{t}{s}\right) XZ^2+s^{6k}B\left(\frac{t}{s}\right)Z^3=0\]
\end{notation}
\begin{remark}
Consider the projection morphism $\psi: W_{A,B} \to \Ps^1$. Then for all $p\in \Ps^1$ the fiber over $p$ is irreducible and the general fiber is an elliptic curve.
If we resolve the singularities of $W_{A,B}$ and then collapse all $-1$ curve then we obtain an elliptic surface $X_{A,B}\to \Ps^1$ \cite[Lecture III.3]{MiES}.
 The collapsing of $-1$ curves is only necessary if the original Weierstrass equation is not minimal, or if both $\deg(A)\leq 4k-4$ and $\deg(B)\leq 6k-6$ hold.
 \end{remark}
 \begin{notation}
We define now the following subset 
\[ U_{m,n}(R)=\left \{(A,B)\in S_{m,n}(R) \left| \begin{array}{c}4A^3+27B^2 \mbox{ is smooth in } K[t] \mbox{ and} \\ \deg(4A^3+27B^2)=\max(3m,2n)\end{array}\right.\right \}.\]
In the following we will  identify $R[t]_{\leq n}\times R[t]_{\leq m}$ with $\As^{m+n+2}_R$. 
\end{notation}

\begin{lemma}
There exists subvarieties $Z'$ and $Z''$ of $\As^{m+n+2}$ such that $S_{m,n}(R)=\As^{m+n+2}(R)\setminus Z'(R)$ and $U_{m,n}(R)=\As^{m+n+2}(R) \setminus Z''(R)$. Moreover, $U_{m,n}(R)\neq \emptyset$.
\end{lemma}
\begin{proof}
The subvariety $Z'$ is defined by the property that all coefficients in $4A^3+27B^2$ vanish. Recall that for every  $(A,B)\in S_{m,n}(R)$ we have that $\deg(4A^3+27B^2)\leq \max (2n,3m)$. Hence $Z''$ is defined by the vanishing of the coefficient of $t^d$ in $4A^3+27B^2$, where $d=\max(2n,3m)$.
It remains to show that $U_{m,n}$ is nonempty: if $2n\geq 3m$ then $(1,t^n)\in U_{m,n}(R)$, if $2n<3m$ then $(t^m,1)\in U_{m,n}(R)$.
\end{proof}
\begin{notation}
For every field $K$ of characteristic zero there is  a subscheme of $\cU_{m,n} \subset \As^{n+m+2}_K$ such that 
$\cU_{m,n}(L)=U_{m,n}(L)$, for every field extension $L/K$. One easily checks that $\cU_{m,n}$ admits a model over $\Z$ such that $\cU_{m,n}(R)=U_{m, n}(R)$ for every integral domain $R$ of characteristic zero. 
\end{notation}

For fixed $(m,n)$  the singularities of $W_{A,B}$ depend on $(A,B)\in S_{m,n}(K)$ and this is an obstruction to construct a family of elliptic surfaces as in Defintion~\ref{defFamily}.
 We will now first show that for $(A,B)\in U_{m,n}(K)$ the Weierstrass model $W_{A,B}$ has at most one singular point, if for one pair $(A,B)$ this model is smooth then all $W_{A,B}$ are smooth and if $W_{A,B}$ is singular then the type of singularity depends only on $(m,n)$.
We will use this to show that $(X_{A,B})_{(A,B)\in U_{m,n}(K)}$ yields a family of elliptic surfaces with section and constant trivial lattice.

Recall that $k=\lceil \max(m/4,n/6) \rceil$.
Let $\alpha=4k-m$ and let $\beta=6k-n$. Note that $\alpha,\beta \geq 0$ and that $\alpha<4$ or $\beta<6$. 

\begin{lemma} Let $(A,B)\in U_{m,n}(K)$ then  $p_g(X_{A,B})=k-1$. Moreover, all singular fibers are of type $I_1$, except possibly for the fiber at $t=\infty$. At $t=\infty$ we have the following fiber depending only on $(\alpha,\beta)$:
\[\begin{array}{|c|c|c|c|c|c|c|c|c|c|c|}
\hline
\alpha & 0 & \geq 0 &\geq 1 &1 &\geq 2 &2& \geq 2& \geq 3 & 3 &\geq 4\\
\beta& \geq 0 & 0 & 1 &\geq 2 &2 & \geq 3 &3 & 4 &\geq 5 & 5\\
\hline
\mbox{fiber-type} &I_0&I_0&II & III& IV &I_0^*&I_0^*&IV^*&III^*&II^*\\
\mbox{singularity}& A_0&A_0&A_0&A_1&A_2 &D_4&D_4&E_6&E_7&E_8\\
\hline
\end{array}\]
\end{lemma}
\begin{proof}
The statement about $p_g(X)$ is \cite[Lemma IV.1.1]{MiES}.
For the fiber over infinity we have $v(c_4)=\alpha$, $v(c_6)=\beta$ and $v(\Delta)=12k-\max(3m,2n)=\min(3\alpha,2\beta)$.
A straight-forward application of Tate's algorithm yields the type of fiber and the type of singularity, see \cite{Tate}.
\end{proof}

For a given $(A,B)\in U_{m,n}$ we constructed a Weierstrass model $W_{A,B}$. This yields a universal family over $\cU_{m,n}$ i.e., a morphism of schemes $\cW_{n,m}\to \cU_{n,m}$.

\begin{proposition}\label{prpConstant} Consider the universal Weierstrass model $\psi:\cW_{m,n}\to \cU_{m,n}$. Then the resolution of the singularity of the generic fiber yields a  family of elliptic surfaces with a section $\cX_{m,n} \to U_{m,n}$ such that for all $b\in |\cU_{m,n}|$ we have $T_\eta=T_b$.
\end{proposition}
\begin{proof}
From the previous lemma it follows that over $\cU_{m,n}$ all Weierstrass models are smooth (if $\alpha=0$ or if $\beta\in \{0,1\}$) or all Weierstrass models have the same type of singularity (if $\alpha>0$ and $\beta>1$).
 In particular, if we manage to resolve the singularities  of the morphism $\cW_{m,n}\to U_{m,n}$ uniformly then we have that the trivial lattice is constant, i.e., $\spa(T_\eta)=T_b$ for all $b\in |\cU_{m,n}|$.

Suppose that the singularity is of type $A_1$ or type $A_2$ then a single blow-up of the point $((X:Y:Z),(s:t))=((0:0:1),(0:1))$  suffices to resolve the singularity, hence this can be done simultaneously for the whole family.
In case the singularity is of type $D_4$ then we have a local equation of type
\[ -y^2+x^3+s^2g(s)x+s^3h(s)=0\]
where $4g(s)^3+27h(s)^2$ does not vanish at $s=0$. If substitute $y=sy, x=sx$ then the strict transform has equation
\[ -y^2+s(x^3+g(s)x+h(s))=0.\]
Since $4g(0)+27h(0)\neq0$ holds, this surface has  $A_1$ singularities at the three points in $y=s=x^3+g(0)x+h(0)=0$. 
 Hence this singularity can be also resolved simultaneously. 
 
 One can proceed similarly for the remaining singularities $E_6,E_7,E_8$. In particular,  if we resolve the singularity at $((0:0:1),(0:1))$ then we obtain a family  $\cX_{m,n}\to \cU_{m,n}$ of elliptic surface over $\Ps^1$ .
\end{proof}

\begin{theorem}  Let $K$ be a finitely generated field over $\Q$. 
Let $m,n$ be positive integers.  Let $\cX_{m,n}$ be the universal elliptic curve over $\cB=\cU_{m,n}$. Then
$E_\eta(K(\eta)(t))$ is finite.
\end{theorem}

\begin{proof}
Fix a pair $(A,B)\in U_{m,n}(K)$. 
Take now a $d\in K^*$. Then the Weierstrass equation of $E^{(d)}_{A,B}$ equals
\[ -YZ^2+X^3+d^2AXZ^2+d^3Z^3\]
In particular, we have that 
$ E^{(d)}_{A,B}\cong E_{d^2A,d^3B}$
as elliptic curves over $K(t)$.

Recall that from $(A,B)\in U_{m,n}(K)$ it follows  that $\Delta(A,B):=4A^3+27B^2$ has degree $\max(3m,2n)$ and is smooth.
Since $\Delta(d^2A,d^3B)=d^6\Delta(A,B)$ holds, also $\Delta(d^2A, d^43B)$ is smooth and of degree $\max(3m,2n)$. Hence $(d^2A,d^3B)\in U_{m,n}(K)$ and therefore $\tw((A,B))=K^*/(K^*)^2$. From Proposition~\ref{prpRankVan} it follows now that $E_\eta(K(\eta)(t))$ is finite.
\end{proof}

\begin{remark} In the following we will use results from \cite{Ser} concerning the density of thin subsets. In these statements the measure on $K[t]_n\times K[t]_n$ is the naive height. Cowan \cite{Cowan}  uses the Mahler measure. Since these measures are equivalent we are free to use these result from \cite{Ser}.
\end{remark}

\begin{corollary}\label{corDens} Suppose $K$ is a number field. Then the set $(A,B)\in S_{m,n}(K)$ such that $E_{A,B}(K(t))$ has rank zero is the complement of a thin set. In particular, this set has density one in $S_{m,n}(K)$.
\end{corollary}
\begin{proof}
Combining Proposition~\ref{prpConstant} with Theorem~\ref{mainThmB} yields that there is a thin subset $Z$ of $U_{m,n}(K)$ such that for all $(A,B)\in U_{m,n}(K)\setminus Z$ we have that the rank of $E_{A,B}(K(t))$ vanishes. Now the complement of $U_{m,n}(K)$ in $K[t]_m\times K[t]_n\cong \As^{m+n+2}$ is a Zariski closed proper subset.  From \cite[Theorem 13.1.3]{Ser} it follows that $U_{m,n}(K)\setminus Z$ has density one  in $K[t]_m \times K[t]_n$, and therefore also density one in the smaller set $S_{m,n}(K)$.
\end{proof}

To prove Cowan's conjecture we have to deal with some subtle points. First of all, Cowan's conjecture is formulated over $\Z$ rather than over $\Q$.
 Moreover, there is a density zero subset of Weierstrass equations whose minimal model is trivial.
Let $s=\lfloor \min(m/4,n/6)\rfloor$.
For an integral domain $R$, let 
\[Z^{tr}_{m,n}(R)=\left\{ (\lambda u^4,\mu u^6)\mid \lambda,\mu\in R, 4\lambda^3+27\mu^2\neq0, u\in R[t]_s\setminus \{0\}\right\}.\]

\begin{proposition} \label{rnkBnd} Let $K$ be a field. Let $m,n$ be positive integers and $k=\lceil \max(m/4,n/6) \rceil$. Let $(A,B)\in S_{m,n}(K)\setminus Z^{tr}_{m,n}(K)$. Then 
\[ \rank E_{A,B}(\oK(t))\leq 10k-2.\]
\end{proposition}
\begin{proof}
For $(A',B')\in U_{m,n}$ we have $p_g(X)=k-1$. Using semi-continuity we find that $p_g(X_{A,B})\leq k-1$.

If $X_{A,B}$ is trivial then a minimal Weierstrass equation of the generic fiber $E_{A,B}/K(t)$ is 
\[ y^2=x^3+\lambda x +\mu.\]
with $\lambda,\mu \in K$.
 Every other short Weierstass equation for $E$, in particular the equation for $W_{A,B}$, is of the form
 \[ y^2=x^3+\lambda u^4 x +\mu u^6.\]
 with $u\in K(t)$. In particular $(A,B)\in Z^{tr}_{m,n}$, which we excluded. 
 
Hence we may apply the Shioda-Tate formula (Proposition~\ref{prpST}) and obtain  $\rank E(\oK(t))\leq h^{1,1}-2$. From \cite[Lemma IV.1.1]{MiES} it follows that $h^{1,1}=10(p_g+1)\leq 10k$. 
\end{proof}

\begin{corollary}Let $m,n$ be positive integers. Then the average rank of $S_{m,n}(\Z)$ is zero.
\end{corollary}

\begin{proof}
From Corollary~\ref{corDens} it follows that there  a thin subset $Z\subset S_{m,n}(\Z)$ such that on $S_{m,n}(\Z)\setminus Z$ the rank is zero. 
Let $Z'=Z\cap Z^{tr}_{m,n}(\Z)$.
Let $Z''=Z\setminus Z'$. By \cite[Theorem 13.1.1]{Ser} both sets $Z'$ and $Z''$ have density zero in $S_{m,n}(\Z)$. 

On $Z''(\Z)$  the Mordell-Weil rank can be bounded by $10k-2$ by Proposition~\ref{rnkBnd}. Since this set has density zero it does not contribute to the average rank.

It is not known whether the rank on $Z'(\Z)$ is bounded or not. However, from \cite{BhaSha} it follows that on $Z^{tr}_{m,n}(\Z)$ the average rank  exists and is finite (cf. \cite[Section 3]{BBD}). Since this set has density zero, it also does not contribute to the average rank.
\end{proof}

We will now generalize Cowan's conjecture to an arbitrary number field $K$. However we cannot control the average rank on $Z^{tr}_{m,n}(\cO_K)$.

\begin{corollary}\label{corCowGen} Let $m,n$ be positive integers. Let $K$ be a number field with ring of integers $\cO_K$. Then the average rank of $S_{m,n}(\cO_K)\setminus Z^{tr}_{m,n}(\cO_K)$ is zero.
\end{corollary}

\begin{proof}
As in the proof of the previous corollary there  a thin subset $Z\subset S_{m,n}(\cO_K)\setminus Z^{tr}_{m,n}(\cO_K)$ such that on $S_{m,n}(\cO_K)\setminus Z$ the rank is zero. 

By Proposition~\ref{rnkBnd} we can bound the Mordell-Weil rank on $Z$ by $10k-2$. 
From \cite[Theorem 13.1.1]{Ser} it follows that the density of $Z$ equals zero.  Hence the average rank is zero.
\end{proof}

This approach to Cowan's conjecture shows that if for a given member in the family there are many twists contained in the family then the average rank is zero. We will now give an example of a family where the average rank is positive. 
\begin{example} For $k$ a positive integer. Consider the following subfamily $\cX$ of $S_{2k,3k}$ given by
\[\{ y^2-x^3+g^3-h^2\mid g\in \Q[t]_{2k},h \in \Q[t]_{3k},\; g^3-h^2 \mbox{ has } 6k \mbox{ distinct factors}\}.   \]
In this case we have that for any $b\in|\cB| $ the set $\tw(b)=\{1 \}.$

In this case one easily checks that every member of $\cX$ has 6 fibers of type $II$, and therefore a torsion-free Mordell-Weil group. Moreover, the Mordell-Weil group contains a non-trivial point $(x,y)=(g,h)$. In particular, the average rank is at least one.
\end{example}

\bibliographystyle{plain}
\bibliography{remke2}

\end{document}